\documentclass[reqno,a4paper]{amsart}
\usepackage{amssymb}
\usepackage{amsmath}
\newcommand{\half}{\frac{1}{2}}
\newcommand{\R}{\mathbb R}
\begin{document} 
\newtheorem{prop}{Proposition}[section]
\newtheorem{Def}{Definition}[section] \newtheorem{theorem}{Theorem}[section]
\newtheorem{lemma}{Lemma}[section] \newtheorem{Cor}{Corollary}[section]

\title[Dirac-Klein-Gordon]{\bf Local well-posedness of the two-dimensional Dirac-Klein-Gordon equations in Fourier-Lebesgue spaces}
\author[Hartmut Pecher]{
{\bf Hartmut Pecher}\\
Fakult\"at f\"ur Mathematik und Naturwissenschaften\\
Bergische Universit\"at Wuppertal\\
Gau{\ss}str.  20\\
42119 Wuppertal\\
Germany\\
e-mail {\tt pecher@uni-wuppertal.de}}
\date{}

\begin{abstract}
The local well-posedness problem is considered for the Dirac-Klein-Gordon system in two space dimensions for data in Fourier-Lebesgue spaces $\widehat{H}^{s,r}$ , where
$\|f\|_{\widehat{H}^{s,r}} = \| \langle \xi \rangle^s \widehat{f}\|_{L^{r'}}$ and  $r$ and $r'$ denote dual exponents. We lower the regularity assumptions on the data with respect to scaling improving the results of d'Ancona, Foschi and Selberg in the classical case $r=2$ . Crucial is the fact that the nonlinearities fulfill a null condition as detected by these authors.
\end{abstract}
\maketitle
\renewcommand{\thefootnote}{\fnsymbol{footnote}}
\footnotetext{\hspace{-1.5em}{\it 2000 Mathematics Subject Classification:} 
35Q40, 35L70 }
\normalsize 
\setcounter{section}{0}
\section{Introduction and main results}

Consider the Cauchy problem for the Dirac-Klein-Gordon 
equations in two space dimensions 
\begin{eqnarray}
\label{0.1}
i(\partial_t + \alpha \cdot \nabla) \psi + M \beta \psi & = & - \phi \beta \psi 
\\
\label{0.2}
(-\partial_t^2 + \Delta) \phi + m\phi & = & - \langle \beta \psi,\psi \rangle
\end{eqnarray}
with (large) initial data
\begin{equation}
\psi(0)  =  \psi_0 \,  , \, \phi(0)  =  \phi_0 \, , \, \partial_t 
\phi(0) = \phi_1 \, .
\label{0.3}
\end{equation}
Here $\psi$ is a two-spinor field, i.e. $\psi : \R^{1+2} \to {\mathcal C}^2$, 
and 
$\phi$ is a real-valued function, i.e. $\phi : \R^{1+2} \to \R$ , 
$m,M 
\in \R$ and $\nabla = (\partial_{x_1} , \partial_{x_2}) $ , $ \alpha \cdot 
\nabla = \alpha^1 \partial_{x_1} + \alpha^2 \partial_{x_2}$ . 
$\alpha^1,\alpha^2, \beta$ are hermitian ($ 2 \times 2$)-matrices satisfying 
$\beta^2 = 
(\alpha^1)^2 = (\alpha^2)^2 = I $ , $ \alpha^j \beta + \beta \alpha^j = 0, $  $ 
\alpha^j \alpha^k + \alpha^k \alpha^j = 2 \delta^{jk} I $ . \\
$\langle \cdot,\cdot \rangle $ denotes the ${\mathcal C}^2$ - scalar product. A 
particular representation is given by  $\alpha^1 = {0\;\,1 \choose 1\;\,0}$ , 
$\alpha^2 =  {0\,-i \choose i\;\,0}$ , $\beta = {1\;\,0\choose0 -1}$.\\
The Cauchy data are assumed to belong to Fourier-Lebesgue spaces:
$\psi_0 \in \widehat{H}^{s,r}$, $\phi_0 \in \widehat{H}^{l,r}$ , $\phi_1 \in \widehat{H}^{l-1,r}$ . Here $\widehat{H}^{s,r}$  , $1\le r < \infty$ , denotes the completion of ${\mathcal S}(\R^2)$ with respect to the norm $\|f\|_{\widehat{H}^{s,r}} = \| \langle \xi \rangle^s \widehat{f}\|_{L^{r'}}$ , where $r$ and $r'$ denote dual exponents and $\widehat{f}$ is the Fourier transform of $f$.

Following \cite{AFS1} it is possible to simplify the system (\ref{0.1}),(\ref{0.2}),(\ref{0.3}) by 
considering the projections onto the one-dimensional eigenspaces of the 
operator 
$-i \alpha \cdot \nabla$ belonging to the eigenvalues $ \pm |\xi|$. These 
projections are given by $\Pi_{\pm}(D)$, where  $ D = 
\frac{\nabla}{i} $ and $\Pi_{\pm}(\xi) = \frac{1}{2}(I 
\pm \frac{\xi}{|\xi|} \cdot \alpha) $. Then $ 
-i\alpha \cdot \nabla = |D| \Pi_+(D) - |D| \Pi_-(D) $ and $ \Pi_{\pm}(\xi) \beta
= \beta \Pi_{\mp}(\xi) $. Defining $ \psi_{\pm} := \Pi_{\pm}(D) \psi$ and 
splitting the function $\phi$ into the sum $\phi = \frac{1}{2}(\phi_+ + 
\phi_-)$, where $\phi_{\pm} := \phi \pm iA^{-1/2} \partial_t \phi $ , $ A:= 
-\Delta+1$ , the Dirac - Klein - Gordon system can be rewritten as
\begin{eqnarray}
\label{*}
(-i \partial_t \pm |D|)\psi_{\pm} & = & -M\beta \psi_{\mp} + \Pi_{\pm}(\phi 
\beta (\psi_+ + \psi_-)) \\
\nonumber
(i\partial_t \mp A^{1/2})\phi_{\pm} & = &\mp A^{-1/2} \langle \beta (\psi_+ + 
\psi_-), \psi+ + \psi_- \rangle \mp A^{-1/2} (m+1)(\phi_+ + \phi_-) . \\
\label {**}
\end{eqnarray}
The initial conditions are transformed into
\begin{equation}
\label{***}
\psi_{\pm}(0) = \Pi_{\pm}(D)\psi_0 \, , \,  \phi_{\pm}(0) = \phi_0 \pm i 
A^{-1/2} \phi_1
\end{equation}

The aim is to minimize the regularity of the data so that local well-posedness holds. Persistence of higher regularity is then a consequence of the fact that the results are obtained by a Picard iteration.

The decisive detection by d'Ancona, Foschi and Selberg \cite{AFS1} was that both nonlinearities satisfy a null condition. This implies that the Cauchy problem in three space dimensions is locally well-posed in the classical case $r=2$ for data $(\psi_0,\phi_0,\phi_1) \in H^s \times H^l \times H^{l-1}$ , where $s>0$ , $l=s+\half$ . This is almost optimal with respect to scaling.

In the case $m=M=0$ the Dirac-Klein-Gordon system is invariant under the rescaling
$$\psi_{\lambda}(t,x)= \lambda^{\frac{3}{2}} \psi(\lambda t,\lambda x) \quad , \quad \phi_{\lambda}(t,x)= \lambda \psi(\lambda t,\lambda x) \,. . $$
Because in N space dimensions
\begin{align*}
\|\psi_{\lambda}(0,\cdot)\|_{\dot{\widehat{H}}^{s,r}} &= \lambda^{\frac{3}{2}} \|\psi(0,\lambda x)\|_{\dot{\widehat{H}}^{s,r}} \sim \lambda^{\frac{3}{2}+s-\frac{N}{r}} \|\psi(0,\cdot)\|_{\dot{\widehat{H}}^{s,r}} \, ,\\
\|\psi_{\lambda}(0,\cdot)\|_{\dot{\widehat{H}}^{s+\half,r}} &= \lambda^{\frac{3}{2}} \|\phi(0,\lambda x)\|_{\dot{\widehat{H}}^{s+\half,r}} \sim \lambda^{\frac{3}{2}+s-\frac{N}{r}} \|\psi(0,\cdot)\|_{\dot{\widehat{H}}^{s+\half,r}}
\end{align*}
the scale-invariant space is
$$(\psi_0,\phi_0,\phi_1) \in \dot{\widehat{H}}^{\frac{N}{r}-\frac{3}{2},r} \times
\dot{\widehat{H}}^{\frac{N}{r}-1,r} \times \dot{\widehat{H}}^{\frac{N}{r}-2,r} \, , $$
where $\|f\|_{\dot{\widehat{H}}^{s,r}} = \| |\xi|^s \widehat{f}\|_{L^{r'}}$ 
Thus in the two-dimensional case the critical spaces are
$$(\psi_0,\phi_0,\phi_1) \in H^{-\half} \times L^2 \times H^{-1} \quad {\mbox for} \, r=2 $$
and $$(\psi_0,\phi_0,\phi_1) \in \widehat{H}^{\half-,1+} \times \widehat{H}^{1+,1+} \times \widehat{H}^{0+,1+} \quad {\mbox for} \, r=1+ \, .$$

We remark that $\dot{\widehat{H}}^{s,r} \sim \dot{H}^{\sigma,2}$ , where $\sigma=s+N(\half-\frac{1}{r})$ in terms of scaling, because $\|\psi_0(\lambda x)\|_{\dot{\widehat{H}}^{s,r}} \sim \lambda^{s-\frac{N}{r}} \|\psi_0\|_{\dot{\widehat{H}}^{s,r}}$ . 

In two space dimensions local well-posedness in the classical case $r=2$ was proven by d'Ancona, Foschi and Selberg \cite{AFS} for $ s > -\frac{1}{5}$ and $\max(\frac{1}{4}-\frac{s}{2},\frac{1}{4}+\frac{s}{2},s) < l < \min(\frac{3}{4}+2s,\frac{3}{4}+\frac{3s}{2},1+s)$ , especially for $(s,l)=(-\frac{1}{5}+,\frac{7}{20})$ and $(s,l)=(0,\frac{1}{4}+)$. Global well-posedness was obtained by Gr\"unrock and the author \cite{GP} for $r=2$ and $s\ge 0$, $l=s+\half$ , using the charge conservation law $\|\psi(t)\|_{L^2} = const$ . This means that there is still a gap concerning LWP between the known results and the minimal regularity predicted by scaling, namely $(s,l)=(-\half,0)$ 
leaving open the problem what happens for $-\half < s < - \frac{1}{5}$ and $0 < l < \frac{7}{20}$ or else $-\half < s <0$ and $0<l\le \frac{1}{4}$. We want to approach this problem by leaving the $L^2$-based data and study the local well-posedness problem for data in $\widehat{H}^{s,r}$-spaces for $1<r<2$ , especially for $r=1+$ . The critical spaces are
$(\psi_0,\phi_0,\phi_1) \in \widehat{H}^{\frac{2}{r}-\frac{3}{2},r} \times
\widehat{H}^{\frac{2}{r}-1,r} \times \widehat{H}^{\frac{2}{r}-2,r} \, , $
i.e. $(\psi_0,\phi_0,\phi_1) \in \in \widehat{H}^{-\half+,r} \times
\widehat{H}^{0+,r} \times \widehat{H}^{-1+,r}$ for $r=1+$ .

Our main Theorem \ref{Thm.0.1} shows that especially for $r=1+$ we may assume $(\psi_0,\phi_0,\phi_1) \in \widehat{H}^{\frac{5}{8}+,r} \times \widehat{H}^{\frac{5}{4}+,r} \times \widehat{H}^{\frac{1}{4}+,r}$ leaving open the interval $\half < s <\frac{5}{8}$ for the spinor and $1<l\le \frac{5}{4}$ . As remarked above in terms of scaling $H^{\frac{5}{8}+,1+} \sim H^{-\frac{3}{8}+}$ and  $H^{\frac{5}{4}+,1+} \sim H^{\frac{1}{4}+}$ . Thus in this sense the gap for the spinor significantly shrinks to $-\half < s \le -\frac{3}{8}$ from $-\half < s \le -\frac{1}{5}$  in the pure $L^2$-case.

This gap phenomenon especially for the low dimensional case $N=2$ also appears for other types of nonlinear wave equations with quadratic nonlinearities. In the three-dimensonal case Gr\"urock \cite{G1} proved for quadratic derivative nonlinear wave equations like $\square u = (\partial u)^2$ an almost optimal well-posedness result in the sense of scaling as $r \to 1$ . This problem was considered in the two-dimensional case by Grigoryan-Tanguay \cite{GT}. For $r=2$ the critical exponent is $s=1$ . The authors prove by use of Strichartz type estimates that $s > \frac{7}{4}$ is sufficient for LWP. For $1 < r < 2$ these authors proved LWP for $s > 1+\frac{3}{2r}$, thus $s > \frac{5}{2}$ for $r=1+$ , which scales like $H^{\frac{3}{2}+}$ , half a derivative away from the critical exponent.

If however a null condition is satisfied for a system of the form $\square u = Q(u,u)$, where $Q$ is one of the null forms of Klainerman, then this gap could be closed by Grigoryan-Nahmod \cite{GN}, who established LWP for $s > \half+\frac{3}{2r}$ , which for $r=1+$ scales like $H^{1+}$  , as desired.

In the classical case $r=2$ it is by now standard to reduce LWP for semilinear wave equations to estimates for the nonlinearities in Bourgain-Klainerman-Machedon spaces $X^{s,b}$ . Gr\"unrock \cite{G} proved that a similar method also works for $1<r<2$ . He also obtained the necessary bilinear estimates for the derivative wave equation by use of the calculations of Foschi-Klainerman \cite{FK}. Later this approach was also used by \cite{GN} and by the author \cite{P} for the Chern-Simons-Higgs and the Chern-Simons-Dirac equations. Using the fact that the nonlinear terms in the Dirac-Klein-Gordon system fulfill a null condition, as was shown by \cite{AFS}, we now combine the estimates in \cite{FK} and a bilinear estimate by \cite{GT}.\\[0.2em]

We now formulate the main result for the DKG system.
\begin{theorem}
\label{Thm.0.1}
Let $1 < r \le 2$ , $\delta > 0$ and $s = s_0+\delta$, $l = l_0+\delta$ . Here $(s_0,l_0) = (\frac{33}{20r} - \frac{41}{40}, \frac{9}{5r} - \frac{11}{20})$ (minimal $s$) and $(s_0,l_0)= (\frac{5}{4r}-\frac{5}{8},\frac{2}{r}-\frac{3}{4})$ (minimal $l$) are admissible. Assume
$$ \psi_0 \in \widehat{H}^{s,r}(\R^2) \, , \, \phi_0 \in \widehat{H}^{l,r}(\R^2) \, , \, \phi_1 \in \widehat{H}^{l-1,r}(\R^2) \, .$$
Then there exists $ T > 0$ , $T=T(\|\psi_0\|_{\widehat{H}^{s,r}} , \|\phi_0\|_{\widehat{H}^{l,r}} , \|\phi_1\|_{\widehat{H}^{l-1,r}})$ such that the DKG system (\ref{0.1}),(\ref{0.2}),(\ref{0.3}) has a unique solution
$$ \psi \in X^r_{s,b,+}[0,T] + X^r_{s,b,-}[0,T] \, , \, \phi \in X^r_{l,b,+}[0,T] + X^r_{l,b,-}[0,T] ,$$
$$ \partial_t \phi \in X^r_{l-1,b,+}[0,T] + X^r_{l-1,b,-}[0,T] , $$
where $b= \frac{1}{r}+$ . This solution satisfies
$$ \psi \in C^0([0,T],\widehat{H}^{s,r}) \, ,  \phi \in C^0([0,T],\widehat{H}^{l,r}) \, ,  \partial_t \phi \in C^0([0,T],\widehat{H}^{l-1,r}) \, . $$
\end{theorem}
The spaces $X^r_{s,b,\pm}$ are generalizations of the Bourgain-Klainerman-Machedon spaces $X^{s,b}$ (for $r=2$). We define $X^r_{s,b\pm}$ as the completion of ${\mathcal S}(\R^{1+2})$ with respect to the norm
$$ \|\phi\|_{X^r_{s,b\pm}} := \| \langle \xi \rangle^s \langle \tau \pm |\xi| \rangle^b \tilde{\phi}(\tau,\xi)\|_{L^{r'}_{\tau \xi}} $$
for $1 \le r \le 2$, $\frac{1}{r} + \frac{1}{r'}=1$ , where $\, \tilde{} \,$ denotes the Fourier transform with respect to space and time. \\
{\bf Remark 1:} By Theorem \ref{Theorem0.3} the solution depends continuously on the data. \\
{\bf Remark 2:} We recover the case $r=2$ with $(s_0,l_0)=(-\frac{1}{5}+,\frac{7}{20}+)$ or $(s_0,l_0)= (0,\frac{1}{4}+)$ from \cite{AFS} and the pair $(s_0,l_0)=(\frac{5}{8}+,\frac{5}{4}+)$ for $r=1+$ .\\
{\bf Remark 3:} By interpolation of the case $r=1+$ with the whole range of pairs $(s,l)$ for $r=2$  from \cite{AFS} (cf. Prop. \ref{Prop.2.2'} below) one obtains further admissible pairs $(s_0,l_0)$ for $1 < r < 2$ . We omit the details.\\[0.5em]

Using the following general local well-posedness theorem (cf. \cite{G}, Theorem 1) we reduce the proof of Theorem \ref{Thm.0.1} to bilinear estimates for the nonlinearities.

\begin{theorem}
\label{Theorem0.3}
Let $N(u)$ be a nonlinear function of degree $\alpha > 0$.
Assume that for given $s \in \R$, $1 < r < \infty$ there exist $ b > \frac{1}{r}$ and $b'\in (b-1,0)$ such that the estimates
$$ \|N(u)\|_{X^r_{s,b',\pm}} \le c \|u\|^{\alpha}_{X^r_{s,b,\pm}} $$
and 
$$\|N(u)-N(v)\|_{X^r_{s,b',\pm}} \le c (\|u\|^{\alpha-1}_{X^r_{s,b,\pm}} + \|v\|^{\alpha-1}_{X^r_{s,b,\pm}}) \|u-v\|_{X^r_{s,b,\pm}} $$
are valid. Then there exist $T=T(\|u_0\|_{\hat{H}^{s,r}})>0$ and a unique solution $u \in X^r_{s,b,\pm}[0,T]$ of the Cauchy problem
$$ \partial_t u \pm iDu = N(u) \quad , \quad u(0) = u_0 \in \hat{H}^{s,r} \, , $$
where $D$ is the operator with Fourier symbol $|\xi|$. This solution is persistent and the mapping data upon solution $u_0 \mapsto u$ , $\hat{H}^{s,r} \to X^r_{s,b,\pm}[0,T_0]$ is locally Lipschitz continuous for any $T_0 < T$.
\end{theorem}

\section{Bilinear estimates}
We start by collecting some fundamental properties of the solution spaces. We rely on \cite{G}. The spaces $X^r_{s,b,\pm} $ with norm  $$ \|\phi\|_{X^r_{s,b\pm}} := \| \langle \xi \rangle^s \langle \tau \pm |\xi| \rangle^b \tilde{\phi}(\tau,\xi)\|_{L^{r'}_{\tau \xi}} $$ for  $1<r<\infty$ are Banach spaces with ${\mathcal S}$ as a dense subspace. The dual space is $X^{r'}_{-s,-b,\pm}$ , where $\frac{1}{r} + \frac{1}{r'} = 1$. The complex interpolation space is given by
$$(X^{r_0}_{s_0,b_0,\pm} , X^{r_1}_{s_1,b_1,\pm})_{[\theta]} = X^r_{s,b,\pm} \, , $$
where $s=(1-\theta)s_0+\theta s_1$, $\frac{1}{r} = \frac{1-\theta}{r_0} + \frac{\theta}{r_1}$ , $b=(1-\theta)b_0 + \theta b_1$ . Similar properties has the space $X^r_{s,b}$ , defined by its norm 
$$ \|\phi\|_{X^r_{s,b}} := \| \langle \xi \rangle^s \langle |\tau| - |\xi| \rangle^b \tilde{\phi}(\tau,\xi)\|_{L^{r'}_{\tau \xi}} \, . $$ 
We also define
$$ X^r_{s,b,\pm}[0,T] = \{ u = U_{|[0,T]\times \R^2} \, : \, U \in X^r_{s,b,\pm} \} $$
with
$$ \|u\|_{X^r_{s,b,\pm}[0,T]} := \inf \{ \|U\|_{X^r_{s,b,\pm}} : U_{|[0,T]\times \R^2} = u \} $$
and similarly $X^r_{s,b}[0,T]$ . \\
If $u=u_++u_-$, where $u_{\pm} \in X^r_{s,b,\pm} [0,T]$ , then $u \in C^0([0,T],\hat{H}^{s,r})$ , if $b > \frac{1}{r}$ .\\[0.2em]

The "transfer principle" in the following proposition, which is well-known in the case $r=2$, also holds for general $1<r<\infty$ (cf. \cite{GN}, Prop. A.2 or \cite{G}, Lemma 1). We denote $ \|u\|_{\hat{L}^p_t(\hat{L}^q_x)} := \|\tilde{u}\|_{L^{p'}_{\tau} (L^{q'}_{\xi})}$ .
\begin{prop}
\label{Prop.0.1}
Let $1 \le p,q \le \infty$ .
Assume that $T$ is a bilinear operator which fulfills
$$ \|T(e^{\pm_1 itD} f_1, e^{\pm_2itD} f_2)\|_{\hat{L}^p_t({\hat L}^q_x)} \lesssim \|f_1\|_{\hat{H}^{s_1,r}} \|f_2\|_{\hat{H}^{s_2,r}} \, .$$
Then for $b > \frac{1}{r}$ the following estimate holds:
$$ \|T(u_1,u_2)\|_{\hat{L}^p_t(\hat{L}^q_x)} \lesssim \|u_1\|_{X^r_{s_1,b,\pm_1}}  \|u_2\|_{X^r_{s_2,b,\pm_2}} \, . $$
\end{prop}

At first we are primarily interested in the case $r=1+$ . Thereafter we obtain the general case $1 < r \le 2$ by bilinear interpolation with the known results for the case $r=2$ .

\begin{prop}
\label{Prop.2.2}
Let $r=1+$ , $l \ge s \ge \frac{5}{8r}$ , $ \half + \frac{3}{4r} < l \le 1 + \frac{1}{4r} $ and $ b > \frac{1}{r}$ .The following estimates apply:
\begin{align}
\label{11}
\| \langle \beta \Pi_{\pm_1}(D) \psi , \Pi_{\pm_2}(D) \psi' \rangle \|_{X^r_{l-1,b-1+}} & \lesssim \|\psi\|_{X^r_{s,b,\pm_1}} \|\psi'\|_{X^r_{s,b,\pm_2}} \, , \\
\label{12}
\| \Pi_{\pm_2}(D)( \phi \beta \Pi_{\pm_1} \psi)\|_{X^r_{s,b-1+,\pm_2}} & \lesssim \|\phi\|_{X^r_{l,b}} \|\psi\|_{X^r_{s,b,\pm_1}} \, .
\end{align}
\end{prop}

By duality (\ref{12}) is equivalent to
$$ \int\int \langle \Pi_{\pm_2}(D)(\phi \beta \Pi_{\pm_1}(D) \psi),\psi' \rangle \,dt \,dx \lesssim \|\phi\|_{X^r_{l,b}} \|\psi\|_{X^r_{s,b,\pm_1}} \|\psi'\|_{X^r_{-s,1-b-,\pm_2}} \, . $$
The left hand side equals
$$ \int \int \phi \langle \beta \Pi_{\pm_1}(D)\psi,\Pi_{\pm_2}(D)\psi' \rangle \, dt \, dx \lesssim \|\phi\|_{X^r_{l,b}} \| \langle \beta \Pi_{\pm_1}(D)\psi,\Pi_{\pm_2}(D) \psi' \rangle \|_{X^{r'}_{-l,-b}} \, , $$
so that (\ref{12}) reduces to
\begin{equation}
\label{12'}
\| \langle \beta \Pi_{\pm_1}(D) \psi,\Pi_{\pm_2}(D) \psi' \rangle \|_{X^{r'}_{-l,-b}} \lesssim  \|\psi\|_{X^r_{s,b,\pm_1}} \|\psi'\|_{X^r_{-s,1-b-,\pm_2}}  \, . 
\end{equation} 

The null structure rests on the following property of the Fourier symbol which is given by the following lemma.
\begin{lemma}
\label{Lemma1} (cf. \cite{AFS}, Lemma 2)
$$\Pi_{\pm_2}(\eta-\xi) \beta \Pi_{\pm_1}(\eta) = \beta \Pi_{\mp_2}(\eta-\xi) \Pi_{\pm_1}(\eta) = O(\angle(\pm_1 \eta,\pm_2(\eta-\xi))) \, ,$$
where $\angle(\eta,\xi)$ denotes the angle between the vectors $\eta$ and $\xi$ .
\end{lemma}

This has the following consequence:
\begin{align}
\label{13}
&| {\mathcal F} (\langle \beta \Pi_{\pm_1}(D) \psi,\Pi_{\pm_2} \psi') \rangle (\tau,\xi)| \\ \nonumber
& \lesssim \int |\langle \beta \Pi_{\pm_1}(\eta) \tilde{\psi}(\lambda,\eta), \Pi_{\pm_2}(\eta-\xi) \tilde{\psi}' (\lambda-\tau,\eta-\xi) \rangle | \, d\lambda \, d\eta \\ \nonumber
&= \int |\langle  \Pi_{\pm_2}(\eta-\xi) \beta \Pi_{\pm_1}(\eta) \tilde{\psi}(\lambda,\eta), \tilde{\psi}' (\lambda-\tau,\eta-\xi) \rangle | \, d\lambda \, d\eta \\
\nonumber
& \lesssim \int \angle(\pm_1 \eta,\pm_2 (\eta - \xi)) |\tilde{\psi}(\lambda,\eta)| \, |\tilde{\psi}'(\lambda-\tau,\eta-\xi)| \, d\lambda \, d\eta \, .
\end{align}

For the angle between two vectors the following elementary estimates apply.
\begin{lemma}
\label{Lemma2}
(cf. \cite{AFS})
\begin{align}
\label{14}
\angle(\eta,\eta-\xi) & \sim \frac{|\xi|^{\half} (|\xi|-||\eta|-|\eta-\xi||)^{\half}}{|\eta|^{\half} |\eta-\xi|^{\half}} \, , \\
\label{15}
\angle(\eta,\xi-\eta) & \sim \frac{(|\eta|+|\xi-\eta|)^{\half} (|\eta|+|\eta-\xi|-|\xi|))^{\half}}{|\eta|^{\half} |\eta-\xi|^{\half}} \, , \\
\label{16}
\angle(\pm_1 \eta, \pm_2(\eta-\xi)) & \lesssim \left( \frac {\langle|\tau|-|\xi|\rangle + \langle \lambda \pm_1 |\eta| \rangle + \langle \lambda-\tau \pm_2|\eta-\xi|\rangle}{\min(\langle \xi \rangle,\langle\eta-\xi \rangle)} \right)^{\half} \, .
\end{align}
\end{lemma}

\begin{proof}[Proof of (\ref{11})] 
By the fractional Leibniz rule the estimate (\ref{11}) follows from 
\begin{equation}
\label{20}
\| \langle \beta \Pi_{\pm_1}(D) \psi , \Pi_{\pm_2}(D) \psi' \rangle \|_{X^r_{0,b-1+}}  \lesssim \|\psi\|_{X^r_{\frac{3}{8r},b,\pm_1}} \|\psi'\|_{X^r_{\frac{5}{8r},b,\pm_2}}
\end{equation}
and the similar estimate
$$
\| \langle \beta \Pi_{\pm_1}(D) \psi , \Pi_{\pm_2}(D) \psi' \rangle \|_{X^r_{0,b-1+}}  \lesssim \|\psi\|_{X^r_{\frac{5}{8r},b,\pm_1}} \|\psi'\|_{X^r_{\frac{3}{8r},b,\mp_2}} \, .$$
We only prove the first one, because the last one is handled in exactly the same way. It is equivalent to
\begin{equation}
\label{17}
\| \langle \beta \Pi_{\pm_1}(D) \psi , \Pi_{\pm_2}(D) \overline{\psi'} \rangle \|_{X^r_{0,b-1+}}  \lesssim \|\psi\|_{X^r_{\frac{3}{8r},b,\pm_1}} \|\psi'\|_{X^r_{\frac{5}{8r},b,\mp_2}} \, .
\end{equation}
The left hand side is bounded by
\begin{align}
\label{18}
&\| {\mathcal F}(\langle \beta \Pi_{\pm_1} \psi,\Pi_{\pm_2} \overline{\psi'} \rangle)\|_{L^{r'}_{\tau \xi}} \\
\nonumber
& = \| \int \langle \beta \Pi_{\pm_1} (\eta) \tilde{\psi}(\lambda,\eta), \Pi_{\pm_2}(\eta-\xi) \tilde{\psi'}(\tau - \lambda,\xi-\eta) \rangle \, d\lambda \, d\eta \|_{L^{r'}_{\tau \xi}} \, .
\end{align}

Let now $\psi(t,x) = e^{\pm_1 itD} \psi_0^{\pm_1}(x) $ and $\psi' = e^{\mp_2 itD} \psi_0'^{\mp_2}(x)$ , so that we obtain
$\tilde{\psi}(\tau,\xi) = c\delta(\tau \mp_1 |\xi|) \widehat{\psi_0^{\pm_1}}(\xi)$ and $\tilde{\psi}'(\tau,\xi) = c\delta(\tau \pm_2 |\xi|) \widehat{\psi_0'^{\mp_2}}(\xi)$ . Then we obtain by Lemma \ref{Lemma1}:
\begin{align}
\nonumber
&\| {\mathcal F}(\langle \beta \Pi_{\pm_1} \psi,\Pi_{\pm_2} \overline{\psi'} \rangle)\|_{L^{r'}_{\tau \xi}} \\
\nonumber
& = c^2 \| \int \langle \Pi_{\pm_2}(\eta-\xi) \beta \Pi_{\pm_1}(\eta) \delta (\lambda\mp_1 |\eta|) \widehat{\psi_0^{\pm_1}}(\eta),\delta(\tau-\lambda \pm_2 |\xi-\eta|) \\
\nonumber 
&\hspace{25em} \widehat{\psi_0'^{\mp_2}}(\xi-\eta) \rangle \, d\eta \, d\lambda \|_{L^{r'}_{\tau\xi}} \\
\label{19}
& \lesssim \| \int \angle(\pm_1 \eta,\pm_2(\eta-\xi)) \delta(\tau \mp_1 |\eta| \pm_2 |\xi-\eta|) |\widehat{\psi_0^{\pm_1}}(\eta)| \, |\widehat{\psi_0'^{\mp_2}}(\xi-\eta)| \, d\eta\|_{L^{r'}_{\tau \xi}} \, .
\end{align}

We now distinguish between the different signs. It suffices to consider the cases
$ \pm_1 = \pm_2 = + $ (hyperbolic case) and $\pm_1 = +$ , $\pm_2 = -$ (elliptic case).\\
{\bf Case $ \pm_1 = \pm_2 = + $.}  Then we obtain from (\ref{14}) and H\"older's inequality:
\begin{align*}
(\ref{19}) & \lesssim \| \int \frac{|\xi|^{\half}(|\xi|-|\tau|)^{\half}}{|\eta|^{\half} |\eta-\xi|^{\half}} \,\delta(\tau-|\eta|+|\xi-\eta|) \,|\widehat{\psi_0^+}(\eta)| \, |\widehat{\psi_0'^-}(\xi-\eta)| \,d\eta \|_{L^{r'}_{\tau \xi}} \\
& \lesssim \sup_{\tau,\xi} I \,\,\|\widehat{D^{\frac{3}{8r}} \psi_0^+}\|_{L^{r'}} \|\widehat{D^{\frac{5}{8r}} \psi_0'^-}\|_{L^{r'}} \, ,
\end{align*}
where
$$ I = |\xi|^{\half} ||\tau|-|\xi||^{\half} (\int \delta(\tau-|\eta|+|\xi-\eta|) |\eta|^{-\frac{3}{8}-\frac{r}{2}} |\eta-\xi|^{-\frac{5}{8}-\frac{r}{2}} \, d\eta)^{\frac{1}{r}} \, . $$
We want to show $\sup_{\tau,\xi} I \lesssim 1$ . \\
Subcase $|\eta|+|\xi-\eta| \le 2|\xi|$ . By \cite{FK}, Prop. 4.5 we obtain
$$\int_{|\eta|+|\xi-\eta| \le 2|\xi|} \delta(\tau-|\eta|+|\xi-\eta|) |\eta|^{-\frac{3}{8}-\frac{r}{2}} |\eta-\xi|^{-\frac{5}{8}-\frac{r}{2}} \, d\eta \sim |\xi|^A ||\tau|-|\xi||^B \, , $$
with $A= \max(\frac{5}{8}+\frac{r}{2},\frac{3}{2}) -1-r = \half - r$ and $B= 1- \max(\frac{5}{8}+\frac{r}{2},\frac{3}{2}) = - \half$ for $r = 1+$ . This implies
$$I^r \lesssim |\xi|^{\frac{r}{2}} ||\tau|-|\xi||^{\frac{r}{2}} |\xi|^{\half-r} ||\tau |-|\xi||^{-\half} = ||\tau|-|\xi||^{\frac{r}{2}-\half} |\xi|^{\half-\frac{r}{2}} \lesssim 1 \, , $$
because $|\tau| \le |\xi|$ . \\[0.2em]
Subcase $|\eta|+|\xi-\eta| \ge 2|\xi|$ . We apply \cite{FK}, Lemma 4.4, and obtain
\begin{align*}
&\int_{|\eta|+|\xi-\eta| \ge 2|\xi|} \delta(\tau-|\eta|+|\xi-\eta|) |\eta|^{-\frac{3}{8}-\frac{r}{2}} |\eta-\xi|^{-\frac{5}{8}-\frac{r}{2}} \, d\eta \\
& \sim (|\xi|^2-\tau^2)^{-\half} \int_2^{\infty} (|\xi|x+\tau)^{-\frac{r}{2}+\frac{3}{8}} (|\xi|x-\tau)^{-\frac{r}{2}+\frac{5}{8}} (x^2-1)^{-\half} \, dx \\
& \sim (|\xi|^2-\tau^2)^{-\half} \int_2^{\infty} (x+\frac{\tau}{|\xi|})^{-\frac{r}{2}+\frac{3}{8}} (x-\frac{\tau}{|\xi|})^{-\frac{r}{2}+\frac{5}{8}} (x^2-1)^{-\half} \, dx \, |\xi|^{1-r} \, .
\end{align*}
The lower limit of the integral can be chosen as $2$ by inspection of the proof of \cite{FK}.
Because $|\tau| \le |\xi|$ the integral is bounded and we obtain
$$ I^r \lesssim |\xi|^{\frac{r}{2}} ||\tau|-|\xi||^{\frac{r}{2}} \frac{|\xi|^{1-r}}{||\tau|-|\xi||^{\half} ||\tau|+|\xi||^{\half}}  \lesssim ||\tau|-|\xi||^{\frac{r}{2}-\half} |\xi|^{\half-\frac{r}{2}} \lesssim 1 \, . $$
{\bf Case $\pm_1=+ \, , \, \pm_2 = -$} . We use (\ref{15}) and H\"older and obtain in the case $|\eta| \ge |\xi-\eta|$ :
\begin{align*}
(\ref{19}) & \lesssim  \| \int \frac{||\tau|-|\xi||^{\half}}{ |\eta-\xi|^{\half}} \,\delta(\tau-|\eta|-|\xi-\eta|) \,|\widehat{\psi_0^+}(\eta)| \, |\widehat{\psi_0'^+}(\xi-\eta)| \,d\eta \|_{L^{r'}_{\tau \xi}} \\
& \lesssim \sup_{\tau,\xi} I \,\,\|\widehat{D^{\frac{3}{8r}} \psi_0^+}\|_{L^{r'}} \|\widehat{D^{\frac{5}{8r}} \psi_0'^+}\|_{L^{r'}} \, ,
\end{align*}
where
$$ I = ||\tau|-|\xi||^{\half} (\int \delta(\tau-|\eta|-|\xi-\eta|) |\eta|^{-\frac{3}{8}} |\eta-\xi|^{-\frac{5}{8}-\frac{r}{2}} \, d\eta)^{\frac{1}{r}} \, . $$
 By \cite{FK}, Lemma 4.3 we obtain
$$\int \delta(\tau-|\eta|-|\xi-\eta|) |\eta|^{-\frac{3}{8}} |\eta-\xi|^{-\frac{5}{8}-\frac{r}{2}} \, d\eta \sim \tau^A ||\tau|-|\xi||^B \, , $$
with $A= \max(\frac{5}{8}+\frac{r}{2},\frac{3}{2}) -(\frac{r}{2}+1) = \half - \frac{r}{2}$ and $B= 1- \max(\frac{5}{8}+\frac{r}{2},\frac{3}{2}) = - \half$ for $r = 1+$ . Using $|\xi| \le \tau$ this implies
$$I^r \lesssim ||\tau|-|\xi||^{\frac{r}{2}} \tau^{\half-\frac{r}{2}} ||\tau |-|\xi||^{-\half}   \lesssim 1 \, . $$
We omit the case $|\eta| \le |\xi-\eta|$ , because it can be treated similarly.

In any case we arrive at the estimate
$$ \| {\mathcal F}(\langle \beta \Pi_{\pm_1} \psi,\Pi_{\pm_2} \overline{\psi'} \rangle)\|_{L^{r'}_{\tau \xi}} \lesssim \|\widehat{D^{\frac{3}{8r}} \psi_0^{\pm_1}}\|_{L^{r'}} \|\widehat{D^{\frac{5}{8r}} \psi_0'^{\mp_2}}\|_{L^{r'}} \, .$$
By the transfer principle Prop. \ref{Prop.0.1} we obtain (\ref{20}), which completes the proof.
\end{proof}

For the proof of (\ref{12'}) we need the following propositions, where we refer to the authors's paper \cite{P} and the Grigoryan-Tanguay paper \cite{GT}.

\begin{prop}
\label{Prop.1.4}
Assume $1 < r \le 2$ ,  $\alpha_0 > \frac{1}{r}-\gamma$ , $\alpha_1 + \alpha_2 > \frac{2}{r}$ , $0 \le \alpha_0 \le \alpha_1,\alpha_2$, $\max(\alpha_1,\alpha_2) \neq \frac{3}{2r}$ , $b \ge \gamma$ ,  and either $\alpha_1+\alpha_2-\alpha_0 > \gamma + \frac{1}{r}$ and $ \gamma \ge \frac{1}{2r}$ , or $\alpha_1+\alpha_2-\alpha_0 \ge \gamma + \frac{1}{r}$ and $ \gamma > \frac{1}{2r}$ . Moreover $ \gamma \ge \max(\alpha_1-\frac{1}{r},\alpha_2-\frac{1}{r})$ , $b > \frac{1}{r}$ .
Then the following estimate holds:
$$ \|uv\|_{X^r_{\alpha_0,\gamma}} \lesssim \|u\|_{X^r_{\alpha_1,b}} \|v\|_{X^r_{\alpha_2,b}} \, . $$
\end{prop}
\begin{proof}
\cite{P}, Proposition 2.6.
\end{proof}

In the case $\gamma =0$ we need the following non-trivial result.
\begin{prop}
\label{Prop.2.3}
Let $ 1 \le r \le 2$ , $\alpha_1,\alpha_2 \ge 0$ , $\alpha_1 + \alpha_2 > \frac{3}{2r}$ , $b_1+b_2 > \frac{3}{2r}$ and $b_1,b_2 > \frac{1}{2r}$ . Then the following estimate holds
$$ \|uv\|_{X^r_{0,0}} \lesssim \|u\|_{X^r_{\alpha_1,b_1}} \|v\|_{X^r_{\alpha_2,b_2}} \, . $$
\end{prop}
\begin{proof}
Selberg \cite{S} proved this in the case $r=2$ . The general case $1 < r \le 2$ was given by Grigoryan-Tanguay \cite{GT}, Prop. 3.1, but in fact the case $r=1$ is also admissible. More precisely the result follows from \cite{GT} after summation over dyadic pieces in a standard way.
\end{proof}

\begin{proof}[Proof of (\ref{12'})]
We apply Lemma \ref{Lemma1} and estimate the angle by (\ref{16}), where we replace  the power $\half$ by  $\frac{1}{2r}$ , which is certainly possible. This allows to reduce (\ref{12'}) by the following estimates:
\begin{align*}
\|u \overline{v}\|_{X^{r'}_{-l,-b+\frac{1}{2r}}} & \lesssim \|u\|_{X^r_{s,b,\pm_1}} \|v\|_{X^{r'}_{-s+\frac{1}{2r},1-b-,\pm_2}} \, , \\
\|u \overline{v}\|_{X^{r'}_{-l,-b+\frac{1}{2r}}} & \lesssim \|u\|_{X^r_{s+\frac{1}{2r},b,\pm_1}} \|v\|_{X^{r'}_{-s,1-b-,\pm_2}} \, , \\
\|u \overline{v}\|_{X^{r'}_{-l,-b}} & \lesssim \|u\|_{X^r_{s,b-\frac{1}{2r},\pm_1}} \|v\|_{X^{r'}_{-s+\frac{1}{2r},1-b-,\pm_2}} \, , \\
\|u \overline{v}\|_{X^{r'}_{-l,-b}} & \lesssim \|u\|_{X^r_{s+\frac{1}{2r},b-\frac{1}{2r},\pm_1}} \|v\|_{X^{r'}_{-s,1-b-,\pm_2}} \, , \\
\|u \overline{v}\|_{X^{r'}_{-l,-b}} & \lesssim \|u\|_{X^r_{s,b,\pm_1}} \|v\|_{X^{r'}_{-s+\frac{1}{2r},1-b-\frac{1}{2r}-,\pm_2}} \, , \\
\|u \overline{v}\|_{X^{r'}_{-l,-b}} & \lesssim \|u\|_{X^r_{s+\frac{1}{2r},b,\pm_1}} \|v\|_{X^{r'}_{-s,1-b-\frac{1}{2r}-,\pm_2}} \, .
\end{align*}
By duality it suffices to prove
\begin{align}
\label{1'}
\|uw\|_{X^r_{s-\frac {1}{2r},b-1+}} & \lesssim \|u\|_{X^r_{s,b}} \|w\|_{X^r_{l,b-\frac{1}{2r}}} \, ,\\
\label{2'}
\|uw\|_{X^r_{s,b-1+}} & \lesssim \|u\|_{X^r_{s+\frac{1}{2r},b}} \|w\|_{X^r_{l,b-\frac{1}{2r}}} \, , \\
\label{3'}
\|uw\|_{X^r_{s-\frac {1}{2r},b-1+}} & \lesssim \|u\|_{X^r_{s,b-\frac{1}{2r}}} \|w\|_{X^r_{l,b}} \,, \\
\label{4'}
\|uw\|_{X^r_{s,b-1+}} & \lesssim \|u\|_{X^r_{s+\frac{1}{2r},b-\frac{1}{2r}}} \|w\|_{X^r_{l,b}} \, ,\\
\label{5'}
\|uw\|_{X^r_{s-\frac {1}{2r},b-1+\frac{1}{2r}+}} & \lesssim \|u\|_{X^r_{s,b}} \|w\|_{X^r_{l,b}} \, ,\\
\label{6'}
\|uw\|_{X^r_{s,b-1+\frac{1}{2r}+}} & \lesssim \|u\|_{X^r_{s+\frac{1}{2r},b}} \|w\|_{X^r_{l,b}} \, .
\end{align}
(\ref{1'}) follows from the fractional Leibniz rule and Prop. \ref{Prop.2.3} , which is fulfilled for $l+ \frac{1}{2r} > \frac{3}{2r} \, \Leftrightarrow \, l > \frac{1}{r}$ and $ 2b-\frac{1}{2r} > \frac{3}{2r} \, \Leftrightarrow b > \frac{1}{r}$ . (\ref{2'}),(\ref{3'}) and (\ref{4'}) follow similarly.

 Next we prove (\ref{6'}). We use Prop. \ref{Prop.1.4} with parameters $\gamma = b-1+\frac{1}{2r}+ = \frac{3}{2r}-1+$, $\alpha_0 = s > \frac{5}{8r} > \frac{1}{r}-\gamma$ , $\alpha_1=s+\frac{1}{2r}$ , $\alpha_2=l$ , so that $\alpha_1 + \alpha_2 - \alpha_0 = l+\frac{1}{2r} >\gamma + \frac{1}{r} = \frac{5}{2r}-1+$ , because by assumption $l > \frac{2}{r}-1$ . Moreover $\alpha_1+\alpha_2 = s + \frac{1}{2r}+l > \frac{2}{r}$ , because by assumption $ s > \frac{5}{8r}$ and $ l > \half + \frac{3}{4r}$ . We also need $ \gamma = \frac{3}{2r}-1+ \ge \max(\alpha_1-\frac{1}{r},\alpha_2-\frac{1}{r}) = \max(s-\frac{1}{2r},l-\frac{1}{r}) $,  because we may assume without loss of generality $ l \le \frac{5}{2r}-1$ and $ s \le \frac{2}{r}-1$ .

Finally we have to prove (\ref{5'}), where it suffices to consider the case $l=\half+\frac{3}{4r}+$. By the fractional Leibniz rule we reduce to the estimates
\begin{align}
\label{5a}
\|uw\|_{X^r_{0,b-1+\frac{1}{2r}+}} & \lesssim \|u\|_{X^r_{\frac{1}{2r},b}} \|w\|_{X^r_{\half+\frac{3}{4r}+,b} }\, , \\
\label{5b}
\|uw\|_{X^r_{0,b-1+\frac{1}{2r}+}} & \lesssim \|u\|_{X^{s,b}} \|w\|_{X^r_{\half+\frac{3}{4r}-s+\frac{1}{2r}+,b}} \, .
\end{align}
Concerning (\ref{5a}) we apply Prop. \ref{Prop.1.4} with $\gamma = 1$ , $\alpha_0 =0$ , $\alpha_1= \frac{1}{2r} $ , $\alpha_2 = 1+\frac{1}{2r}+$, so that $\alpha_1+\alpha_2 = 1+\frac{1}{r}+  > \frac{2}{r}$ and $\alpha_1+\alpha_2-\alpha_0 = 1+\frac{1}{r}+> \gamma + \frac{1}{r}$ . Thus
$$ \|uw\|_{X^r_{0,1}} \lesssim \|u\|_{X^r_{\frac{1}{2r},b}} \|w\|_{X^r_{1+\frac{1}{2r}+,b}} \, . $$
Moreover we apply Prop. \ref{Prop.2.3}  with $\alpha_1=\frac{1}{2r}$ , $\alpha_2=\frac{1}{r}+$ , $b_1=b_2=b$ , thus $\alpha_1+\alpha_2 >\frac{3}{2r}$ and $b_1+b_2 >\frac{3}{2r}$ .Thus
$$ \|uw\|_{X^r_{0,0}} \lesssim \|u\|_{X^r_{\frac{1}{2r},b}} \|w\|_{X^r_{\frac{1}{r}+,b}} \, . $$
Interpolation between these estimates implies
$$\|uw\|_{X^r_{0,\half}} \lesssim \|u\|_{X^r_{\frac{1}{2r},b}} \|w\|_{X^r_{\half+\frac{3}{4r}+,b}} \, ,$$
which proves (\ref{5a}). 

Concerning (\ref{5b}) we argue similarly. We obtain
$$ \|uw\|_{X^r_{0,1}} \lesssim \|u\|_{X^r_{s,b}} \|w\|_{X^r_{1+\frac{1}{r}-s,b}}   $$
and
$$ \|uw\|_{X^r_{0,0}} \lesssim \|u\|_{X^r_{s,b}} \|w\|_{X^r_{\frac{3}{2r}-s+,b}}  \, , $$
so that interpolation implies
$$ \|uw\|_{X^r_{0,\half}} \lesssim \|u\|_{X^r_{s,b}} \|w\|_{X^r_{\half+\frac{5}{4r}-s+,b}} \,,  $$
which proves (\ref{5b}) and completes the proof of (\ref{12'}).
\end{proof}
{\bf Remark:} It is (\ref{5'}) which prevents the optimal choice $s=\half+$ , $l=1+$ in the case $r=1+$ . All the other estimates which are necessary for the proof of our main theorem are valid for this choice.\\[0.2em]

The bilinear estimates in the case $r=2$ by \cite{AFS}, Theorem 1 are given by the following proposition.
\begin{prop}
\label{Prop.2.2'}
Let $r=2$. The estimates (\ref{11}) and (\ref{12}) are fulfilled in the region
$$ s > - \frac{1}{5} \quad, \quad \max(\frac{1}{4}-\frac{s}{2},\frac{1}{4}+\frac{s}{2},s) < l < \min(\frac{3}{4}+2s,\frac{3}{4}+\frac{3s}{2},1+s) \, . $$
\end{prop}

The admissible pairs $(s,l)$ in the general case $1 < r \le 2$ are now obtained by bilinear interpolation between the estimates in Prop. \ref{Prop.2.2} and Prop. \ref{Prop.2.2'}. Because we are mainly interested in the minimal possible choice of $s$ and $l$ we concentrate on the following result for simplicity.

\begin{prop}
\label{Prop.2.2''}
Let $ 1 < r \le 2 $ , $ b = \frac{1}{r}+$ and $\delta > 0$ . The estimates (\ref{11}) and (\ref{12}) are fulfilled in the cases $(s,l) = (\frac{33}{20r} - \frac{41}{40}+\delta , \frac{9}{5r} - \frac{11}{20}+\delta)$ (minimal $s$) and $(s,l)= (\frac{5}{4r}-\frac{5}{8}+\delta,\frac{2}{r}-\frac{3}{4}+\delta)$ (minimal $l$).
\end{prop}
\begin{proof}
We interpolate between the pair $(s,l)=(\frac{5}{8}+,\frac{5}{4}+)$ in the case $r=1+$ on the one hand and the pairs $(s,l) = (-\frac{1}{5}+,\frac{7}{20})$ and  $(s,l)=(0,\frac{1}{4}+)$ in the case $r=2$ on the other hand to obtain the first and second claimed pair $(s,l)$, respectively. We concentrate on the second pair . Let $\delta > 0$ be given and $s= \frac{5}{4r}-\frac{5}{8}+\delta$, $l=\frac{2}{r}-\frac{3}{4}+\delta$ . If $r > 1$ is sufficiently close to 1 we have $\delta > \frac{5}{4}-\frac{5}{4r}$ , so that $\delta = \frac{5}{4}-\frac{5}{4r} + \omega$ , where $\omega > 0$ . For $\omega = 0+$ we obtain $s= \frac{5}{8}+$ and $l= \half+\frac{3}{4r}+$ . In this case the estimates (\ref{11}) and (\ref{12}) are satisfied. By  the fractional Leibniz rule this is also true for every $\omega > 0$ , thus for the given $\delta$ and r close enough to 1. Bilinear interpolation with the case $s= \delta$ and $l=\frac{1}{4}+\delta$ in the case $r=2$ implies the estimates (\ref{11}) and (\ref{12}) for the given pair $(s,l)$ in the whole range $1 < r \le 2$.
\end{proof}


\begin{thebibliography}{999}
\bibitem[1]{AFS} P. d'Ancona, D. Foschi, and S. Selberg: {\sl Local well-posedness below the charge norm for the Dirac-Klein-Gordon system in two space dimensions}. 
J. Hyperbolic Diff. Equns. 4 (2007), no. 2, 295–330
\bibitem[2]{AFS1}
P. d'Ancona, D. Foschi, and S. Selberg, {\sl Null structure and almost optimal local regularity for the Dirac-Klein-Gordon system}, J. Eur. Math. Soc.  (2007), no. 4, 877-899 
\bibitem[3]{FK} D. Foschi and S. Klainerman: {\sl Bilinear space-time estimates for homogeneous wave equations.} Ann. Sc. ENS. 4. serie, 33 (2000), 211-274
\bibitem[4]{GN} V. Grigoryan and A. Nahmod: {\sl Almost critical well-posedness for nonlinear wave equation with $Q_{\mu \nu}$ null forms in 2D.} Math. Res. Letters 21 (2014), 313-332
\bibitem[5]{GT} V. Grigoryan and A. Tanguay: {\sl Improved well-posedness for the quadratic derivative nonlinear wave equation in 2D.} J. Math. Analysis Appl. 475 (2019), 1578-1595
\bibitem[6]{G} A. Gr\"unrock: {\sl An improved local well-posedness result for the modified KdV equation.} Int. Math. Res. Not. (2004), no.61, 3287-3308
\bibitem[7]{G1} A. Gr\"unrock: {\sl On the wave equation with quadratic nonlinearities in three space dimensions.} J. Hyperbolic Diff. Equ. 8 (2011), 1-8
\bibitem[8]{GP} A. Gr\"unrock and H. Pecher: {\sl Global solutions for the
Dirac-Klein-Gordon system in two space dimensions}. Comm. Part. Diff. Equs. 35 (2009), 89-112
\bibitem[9]{P} H. Pecher: {\sl The Chern-Simons-Higgs and the Chern-Simons-Dirac equations in Fourier-Lebesgue spaces}.   Discrete Contin. Dyn. Syst. 39 (2019), 4875–4893.
\bibitem[10]{S} S. Selberg: {\sl Bilinear Fourier restriction estimates related to the 2D wave equation.} Adv. Diff. Equ. 16 (2011), 667-690
 \end{thebibliography}
\end{document}